\documentclass[11pt, a4paper, reqno]{amsart}

\usepackage[utf8]{inputenc}
\usepackage[english]{babel}
\usepackage{amsmath, amssymb, amsthm, amsfonts}
\usepackage{mathtools}
\usepackage{geometry}
\usepackage{listings}
\usepackage{xcolor}
\usepackage{hyperref}
\usepackage{graphicx}
\usepackage{array}
\usepackage{longtable}
\usepackage{booktabs}

\definecolor{codegreen}{rgb}{0,0.6,0}
\definecolor{codegray}{rgb}{0.5,0.5,0.5}
\definecolor{codepurple}{rgb}{0.58,0,0.82}
\definecolor{backcolour}{rgb}{0.95,0.95,0.92}

\lstdefinestyle{mystyle}{
    backgroundcolor=\color{backcolour},
    commentstyle=\color{codegreen},
    keywordstyle=\color{magenta},
    numberstyle=\tiny\color{codegray},
    stringstyle=\color{codepurple},
    basicstyle=\ttfamily\footnotesize,
    breakatwhitespace=false,
    breaklines=true,
    captionpos=b,
    keepspaces=true,
    numbers=left,
    numbersep=5pt,
    showspaces=false,
    showstringspaces=false,
    showtabs=false,
    tabsize=2
}
\newtheorem{theorem}{Theorem}[section]

\newtheorem{proposition}[theorem]{Proposition}

\newtheorem{remark}[theorem]{Remark}

\newcommand{\BCC}{\mathrm{BCC}}

\newcommand{\R}{\mathbb{R}}

\title[Local minimality of BCC]{Local minimality of the truncated octahedron \\ for the isoperimetric problem on parallelohedra}

\author{Annalisa Cesaroni}
\address{Dipartimento di Matematica ``Tullio Levi-Civita'', Universit\`{a} di Padova, Via Trieste 63, 35131 Padova, Italy}
\email{annalisa.cesaroni@unipd.it}
\author{Matteo Novaga}
\address{Dipartimento di Matematica, Universit\`{a} di Pisa, Largo Bruno Pontecorvo 5, 56127 Pisa, Italy}
\email{matteo.novaga@unipi.it}

\date{\today}

\begin{document}

\maketitle

\begin{abstract}
We investigate the isoperimetric problem for the Voronoi cells of three-dimensional lattices. Using Selling parameters, we derive an explicit closed formula for the scale-invariant isoperimetric quotient $F$ in terms of six non-negative variables. We then analyse the local behaviour of $F$ at the most relevant lattice configurations: we prove that the body-centered cubic lattice (BCC) is a strict local minimiser of $F$ at fixed volume, whereas  the face-centered cubic lattice (FCC) and the simple cubic lattice (SC) are not local minimisers. Then, we consider a family of lattices which interpolates between BCC and FCC, showing that BCC is the global minimiser of $F$ restricted to this family.
\end{abstract}

\tableofcontents

\section{Introduction}
The so-called \textit{Kelvin problem} asks how the space can be partitioned into cells of equal volume with the least possible surface area. In 1887 Lord Kelvin proposed a structure based on the bitruncated cubic honeycomb, whose cells are slightly curved truncated octahedra. 

In 1993 Weaire and Phelan discovered a more efficient foam, the Weaire--Phelan structure, with about $0.3\%$ lower surface area. This disproved Kelvin's original conjecture but left open a constrained version of the problem: find the optimal set, which tiles the space by translations (see \cite{CN}). An even more restricted problem is to look for the optimal parallelohedron, i.e., a convex polyhedron that tiles space by translations. The \textit{Truncated Octahedron Conjecture}, originally attributed to Bezdek, states that among these space-filling polyhedra of unit volume the truncated octahedron, which is the Voronoi cell of the body-centered cubic (BCC) lattice, has the minimal surface area for fixed volume (see \cite{bezdek_sphere_packings}, \cite[Conjecture 1]{Langi}).

To approach this geometric conjecture analytically, we utilize an algebraic representation of the lattices. Following Barnes and Sloane \cite{B, BS}, any three-dimensional lattice can be represented by six non-negative Selling parameters $\rho_{ij}$ ($0\le i<j\le3$). The volume of the parallelohedron associated to the parameters 
$\boldsymbol{\rho}=(\rho_{01},\rho_{02},\rho_{03},\rho_{12},\rho_{13},\rho_{23})$ will be denoted by $V=\sqrt{\det A(\boldsymbol{\rho})}$, and its surface area by $A_T(\boldsymbol{\rho})$. We then define the isoperimetric quotient
\begin{equation}
F(\boldsymbol{\rho})=\frac{A_T(\boldsymbol{\rho})}{(\det A)^{1/3}}.
\label{eq:F}
\end{equation}
Minimising $F(\boldsymbol{\rho})$ is equivalent to minimising the surface area $A_T(\boldsymbol{\rho})$ 
for fixed volume $V(\boldsymbol{\rho})$.

Understanding the minimal surfaces of these lattice cells is not only a fundamental mathematical pursuit but also holds physical significance. In fields such as crystallography and materials science, the mathematical stability of these geometric structures mirrors physical phenomena, such as the structural behaviour and isoperimetric quotients of three-dimensional Voronoi tessellations generated by crystals perturbed by Gaussian noise.

In this paper we first derive an explicit formula for $F$ in terms of the six Selling parameters, compute its gradient and Hessian symbolically, and study the local behaviour of the most relevant lattice configurations. In particular, we prove that the BCC lattice is a strict local minimiser, while the simple cubic (SC) lattice is not stationary, and the face-centered cubic (FCC) lattice is stationary but not a local minimum. We then compare these exact values with the numerical simulations in \cite{Lucarini2008}.

Then, we turn to global minimality restricted to a low-dimensional stratum, namely we consider the family of lattices that can be expressed with only two values assigned to the six Selling parameters, which interpolates between BCC and FCC, and we prove that the BCC lattice is the global minimiser of $F$ in this family.

\smallskip
The paper is organised as follows: 
Section 2 establishes the notation and preliminary definitions regarding the Selling parameters and Gram matrices. 
Section 3 contains our primary analysis of parallelohedra, including the local minimality proofs for the BCC, FCC, and SC lattices. 
Section 4 compares our theoretically derived exact values with existing numerical simulations of physical foams. 
Finally, Section 5 establishes the global minimality of the BCC lattice on a specific family of lattices which includes BCC and FCC.

\medskip 

\paragraph{\bf Acknowledgements.}  The authors are members of INDAM-GNAMPA. M.N. acknowledges support from the MIUR Excellence Department Project awarded to the Department of Mathematics, University of Pisa, CUP I57G22000700001. 

\section{Notation and preliminary definitions}

A lattice $\Lambda\subset\mathbb R^{3}$ can be described by a basis
$(b_{1},b_{2},b_{3})$ of $\mathbb R^{3}$ and the corresponding basis matrix
\[
B=\bigl[b_{1}\ \ b_{2}\ \ b_{3}\bigr]\in\mathbb R^{3\times 3},
\qquad\text{so that}\qquad
\Lambda = B\,\mathbb Z^{3}.
\]
Thus every lattice vector is an integer combination of the basis vectors,
\[
v=x_{1}b_{1}+x_{2}b_{2}+x_{3}b_{3}=Bx,
\qquad x=(x_{1},x_{2},x_{3})^{T}\in\mathbb Z^{3}.
\]
Following \cite{B, BS},
$\Lambda$ is encoded by the \emph{Gram matrix} of the basis,
\[
A=\bigl(\langle b_{i},b_{j}\rangle\bigr)_{1\le i,j\le 3}=B^{T}B,
\]
which is symmetric and positive definite. In particular, the squared Euclidean
norm of a lattice vector $v=Bx$ can be written as a quadratic form in the integer
coordinates $x$:
\[
\|v\|^{2}=\|Bx\|^{2}=(Bx)\cdot(Bx)=x^{T}(B^{T}B)\,x=x^{T}Ax.
\]
Equivalently,
\[
x^{T}Ax=\sum_{i=1}^{3}a_{ii}x_{i}^{2}+2\sum_{1\le i<j\le 3}a_{ij}x_{i}x_{j},
\]
so the coefficients $a_{ij}=\langle b_{i},b_{j}\rangle$ record the squared lengths
($a_{ii}=\|b_{i}\|^{2}$) and the mutual angles ($a_{ij}=\|b_{i}\|\,\|b_{j}\|\cos\theta_{ij}$).
Therefore, studying the set of squared distances in $\Lambda$ reduces to studying
the values of the quadratic form
\[
f(x)=x^{T}Ax \qquad \text{for } x\in\mathbb Z^{3}.
\]

To explicitly study the geometry of the cell and the metric of the lattice, it is convenient to switch from the standard Gram matrix entries $a_{ij} = \langle b_i, b_j \rangle$ to the Selling--Conway parameters \cite{BS}. We introduce a strictly dependent fourth vector $b_0 = -b_1 - b_2 - b_3$, so that $\sum_{i=0}^3 b_i = 0$. The six parameters are then defined as the negative inner products of these four vectors:
\[
\rho_{ij} = -\langle b_i, b_j \rangle \qquad \text{for } 0 \le i < j \le 3.
\]
This naturally defines the non-diagonal entries of the Gram matrix as $a_{ij} = -\rho_{ij}$ for $1 \le i < j \le 3$. Furthermore, expanding the relation $0 = \langle b_i, \sum_{j=0}^3 b_j \rangle$ yields the squared lengths of the basis vectors as sums of these parameters:
\[
a_{ii} = \langle b_i, b_i \rangle = -\sum_{j \neq i, \, j=0}^3 \langle b_i, b_j \rangle = \sum_{j \neq i, \, j=0}^3 \rho_{ij}.
\]
By collecting these six values into a parameter vector $\boldsymbol{\rho} = (\rho_{01}, \rho_{02}, \rho_{03}, \rho_{12}, \rho_{13}, \rho_{23})$, the Gram matrix $A=(a_{ij})_{1\le i,j\le 3}$ takes the highly symmetric form:
\begin{equation}
A(\boldsymbol{\rho}) = \begin{pmatrix}
\rho_{01}+\rho_{12}+\rho_{13} & -\rho_{12} & -\rho_{13} \\
-\rho_{12} & \rho_{02}+\rho_{12}+\rho_{23} & -\rho_{23} \\
-\rho_{13} & -\rho_{23} & \rho_{03}+\rho_{13}+\rho_{23}
\end{pmatrix}.
\label{eq:A}
\end{equation}
For brevity we set $a=\rho_{01},\; b=\rho_{02},\; c=\rho_{03},\; d=\rho_{12},\; e=\rho_{13},\; f=\rho_{23}$.

In the generic interior case, i.e. when all six Selling parameters are strictly positive, the Voronoi cell has $24$ vertices and $14$ faces (8 hexagons and 6 quadrilaterals). On boundary strata some vertices and faces merge, yielding the degenerations discussed later in the paper.
The $24$ vertices are obtained as intersections of the planes $F_i$, $F_{ij}$ and $F_{ijk}$ defined in the classical reduction theory; in the $y$‑coordinates (where $y = A x$) they are linear functions of the parameters. For completeness we list a generating set of vertices, the remaining ones are obtained by central inversion and index permutations (see for instance \cite{BS}).

\begin{small}
\[
\begin{array}{ll}
v_{102}=\frac{1}{2}\big(a+d+e,\; f-b-d,\; -c-e-f\big), &
v_{120}=\frac{1}{2}\big(a+d+e,\; b-d+f,\; -c-e-f\big), \\[4pt]
v_{103}=\frac{1}{2}\big(a+d+e,\; -b-d-f,\; f-c-e\big), &
v_{130}=\frac{1}{2}\big(a+d+e,\; -b-d-f,\; c+f-e\big), \\[4pt]
v_{123}=\frac{1}{2}\big(a+d+e,\; b-d+f,\; c-e-f\big), &
v_{132}=\frac{1}{2}\big(a+d+e,\; b-d-f,\; c-e+f\big), \\[4pt]
v_{201}=\frac{1}{2}\big(-a-d+e,\; b+d+f,\; -c-e-f\big), &
v_{210}=\frac{1}{2}\big(a-d+e,\; b+d+f,\; -c-e-f\big), \\[4pt]
v_{213}=\frac{1}{2}\big(a-d+e,\; b+d+f,\; c-e-f\big), &
v_{231}=\frac{1}{2}\big(a-d-e,\; b+d+f,\; c+e-f\big), \\[4pt]
v_{230}=\frac{1}{2}\big(-a-d-e,\; b+d+f,\; c+e-f\big), &
v_{203}=\frac{1}{2}\big(-a-d-e,\; b+d+f,\; -c+e-f\big), \\[4pt]
v_{312}=\frac{1}{2}\big(a+d-e,\; b-d-f,\; c+e+f\big), &
v_{321}=\frac{1}{2}\big(a-d-e,\; b+d-f,\; c+e+f\big).
\end{array}
\]
\end{small}

The remaining vertices are obtained by central inversion and index permutations:
\begin{align*}
v_{302} &= -v_{120},\quad v_{320} = -v_{102},\quad v_{301} = -v_{210},\quad v_{310} = -v_{201},\\
v_{023} &= -v_{132},\quad v_{032} = -v_{123},\quad v_{013} = -v_{231},\quad v_{031} = -v_{213},\\
v_{012} &= -v_{321},\quad v_{021} = -v_{312}.
\end{align*}

For a polygon with vertices $p_1,\dots,p_n$ (in $y$‑coordinates) the vector area is
\begin{equation}
\mathbf{V} = \frac12\sum_{k=1}^{n} p_k\times p_{k+1},\qquad p_{n+1}=p_1.
\end{equation}
The physical area under the metric induced by $A$ is
\begin{equation}
\operatorname{Area} = \frac{\sqrt{\mathbf{V}^T A \mathbf{V}}}{\sqrt{\det A}}.
\label{eq:area}
\end{equation}
To justify \eqref{eq:area}, we can write the Gram matrix as $A=B^TB$, where the columns of $B$ are the basis vectors of the lattice. If $x$ denotes the lattice coordinates and $v=Bx$ the corresponding Euclidean point, then the $y$-coordinates satisfy
\[
y=Ax=B^TBx=B^Tv,
\]
so the passage from $y$-space to Euclidean space is given by the linear map
\[
v=B^{-T}y.
\]
Let now a face be identified by vertices $p_1,\dots,p_n$ in $y$-coordinates, and let
\[
\mathbf V=\frac12\sum_{k=1}^n p_k\times p_{k+1}
\]
be its vector area in $y$-space. Under a linear map $M$, area vectors transform by
\[
(Mu)\times(Mv)=\det(M)\,M^{-T}(u\times v).
\]
Taking $M=B^{-T}$, the Euclidean area vector of the face is therefore
\[
\mathbf V_{\mathrm{phys}}=\det(B^{-T})\,B\,\mathbf V,
\]
and hence
\[
\|\mathbf V_{\mathrm{phys}}\|^2
=\det(B^{-T})^2\,\mathbf V^T B^TB\,\mathbf V
=\frac{\mathbf V^TA\mathbf V}{(\det B)^2}
=\frac{\mathbf V^TA\mathbf V}{\det A},
\]
since $\det A=(\det B)^2$. Taking square roots yields \eqref{eq:area}.

The Voronoi cell has $6$ quadrilateral and $8$ hexagonal faces. Grouping opposite faces gives $7$ independent vector areas:
\begin{align*}
\mathbf{V}_{12} &=\tfrac12\big(v_{120}\times v_{123}+v_{123}\times v_{213}+v_{213}\times v_{210}+v_{210}\times v_{120}\big),\\
\mathbf{V}_{13} &=\tfrac12\big(v_{130}\times v_{132}+v_{132}\times v_{312}+v_{312}\times v_{310}+v_{310}\times v_{130}\big),\\
\mathbf{V}_{23} &=\tfrac12\big(v_{230}\times v_{231}+v_{231}\times v_{321}+v_{321}\times v_{320}+v_{320}\times v_{230}\big),\\
\mathbf{V}_{1}  &=\tfrac12\big(v_{102}\times v_{120}+v_{120}\times v_{123}+v_{123}\times v_{132}+v_{132}\times v_{130}+v_{130}\times v_{103}+v_{103}\times v_{102}\big),\\
\mathbf{V}_{2}  &=\tfrac12\big(v_{201}\times v_{210}+v_{210}\times v_{213}+v_{213}\times v_{231}+v_{231}\times v_{230}+v_{230}\times v_{203}+v_{203}\times v_{201}\big),\\
\mathbf{V}_{3}  &=\tfrac12\big(v_{302}\times v_{320}+v_{320}\times v_{321}+v_{321}\times v_{312}+v_{312}\times v_{310}+v_{310}\times v_{301}+v_{301}\times v_{302}\big),\\
\mathbf{V}_{0}  &=\tfrac12\big(v_{012}\times v_{021}+v_{021}\times v_{023}+v_{023}\times v_{032}+v_{032}\times v_{031}+v_{031}\times v_{013}+v_{013}\times v_{012}\big).
\end{align*}

Denote by $Q_* = \mathbf{V}_*^T A \mathbf{V}_*$ (a quartic polynomial in $a,b,c,d,e,f$). Using \eqref{eq:area}, the total area is
\begin{equation}
A_T = \frac{2}{\sqrt{\det A}}\sum\limits_{*\in\{12,13,23,1,2,3,0\}} \sqrt{Q_*}.
\label{eq:AT}
\end{equation}

Finally,
\begin{equation}
F(\boldsymbol{\rho})=\frac{A_T}{(\det A)^{1/3}}.
\label{eq:Fdef}
\end{equation}

From \eqref{eq:AT} and \eqref{eq:Fdef} we have
\[
F(a,b,c,d,e,f)
=\frac{A_T}{(\det A)^{1/3}}
=\frac{2}{(\det A)^{5/6}}\sum_{*\in\{12,13,23,1,2,3,0\}}\sqrt{Q_*},
\qquad
Q_*=\mathbf V_*^{T}A\mathbf V_*.
\]
 
\noindent\emph{Determinant of $A$.}
We now compute $\det A$ explicitly from \eqref{eq:A}.
Expanding along the first row gives
\begin{align*}
\det A
&=(a+d+e)\bigl[(b+d+f)(c+e+f)-f^2\bigr]
      -d^2(c+e+f)-2def-e^2(b+d+f).
\end{align*}
Since
\[
(b+d+f)(c+e+f)-f^2 = bc+be+bf+cd+de+df+cf+ef,
\]
we can write
\begin{align*}
\det A
&=(a+d+e)\,(bc+be+bf+cd+de+df+cf+ef) \\
&\quad -\bigl(cd^2+d^2e+d^2f\bigr)-\bigl(be^2+de^2+e^2f\bigr)-2def.
\end{align*}
Expanding the product as $a(\cdots)+d(\cdots)+e(\cdots)$ yields
\begin{align*}
(a+d+e)\,(&bc+be+bf+cd+de+df+cf+ef) \\
&=\underbrace{(abc+abe+abf+acd+acf+ade+adf+aef)}_{\text{$a$--part}} \\
&\quad +\underbrace{(bcd+bde+bdf+cdf)+\bigl(cd^2+d^2e+d^2f\bigr)+def}_{\text{$d$--part}} \\
&\quad +\underbrace{(bce+bef+cde+cef)+\bigl(be^2+de^2+e^2f\bigr)+def}_{\text{$e$--part}}.
\end{align*}
Subtracting $\bigl(cd^2+d^2e+d^2f\bigr)$, $\bigl(be^2+de^2+e^2f\bigr)$ and $2def$ cancels the square terms and the two occurrences of $def$,
so we obtain
\[
\det A=
\begin{aligned}[t]
&abc+abe+abf+acd+acf+ade+adf+aef \\
&\quad + bcd+bce+bde+bdf+bef+cde+cdf+cef.
\end{aligned}
\]

\smallskip
\noindent\emph{Vector areas of the faces.}
Vector area is translation invariant, and for a parallelogram generated by edge
vectors $u,v$ one has $\mathbf V=u\times v$.
More generally, if a centrally symmetric hexagon has successive edge vectors
$u,v,w,-u,-v,-w$, then
\begin{equation}\label{eq:zonogon}
\mathbf V = u\times v + u\times w + v\times w.
\end{equation}
Indeed, after translating so that the first vertex is at the origin, the vertices
are $p_0=0$, $p_1=u$, $p_2=u+v$, $p_3=u+v+w$, $p_4=v+w$, $p_5=w$.
A direct application of $\mathbf V=\tfrac12\sum_{k=0}^{5}p_k\times p_{k+1}$
(with $p_6=p_0$) gives \eqref{eq:zonogon}.

For $F_{13}$ the vertices are $v_{130},v_{132},v_{312},v_{310}$ and one checks
\[
v_{132}-v_{130}=(0,b,0),\qquad v_{312}-v_{132}=(-e,0,e).
\]
Hence $F_{13}$ is a parallelogram and
\[
\mathbf V_{13}=(0,b,0)\times(-e,0,e)=be\,(1,0,1).
\]
Similarly,
\[
v_{231}-v_{230}=(a,0,0),\qquad v_{321}-v_{231}=(0,-f,f),
\]
so
\[
\mathbf V_{23}=(a,0,0)\times(0,-f,f)=-af\,(0,1,1),
\]
and
\[
v_{123}-v_{120}=(0,0,c),\qquad v_{213}-v_{123}=(-d,d,0),
\]
so
\[
\mathbf V_{12}=(0,0,c)\times(-d,d,0)=-cd\,(1,1,0).
\]
Note that  the signs depend on the chosen cyclic ordering and are irrelevant for $Q_*=\mathbf V_*^TA\mathbf V_*$.

Consider $F_1$ with cyclic vertices $v_{102},v_{120},v_{123},v_{132},v_{130},v_{103}$.
The successive edge vectors are
\[
u_b:=v_{120}-v_{102}=(0,b,0),\quad
u_c:=v_{123}-v_{120}=(0,0,c),\quad
u_f:=v_{132}-v_{123}=(0,-f,f),
\]
and then $-u_b,-u_c,-u_f$. Therefore $F_1$ is a zonogon generated by $u_b,u_c,u_f$,
and \eqref{eq:zonogon} yields
\[
\mathbf V_{1}
= u_b\times u_c + u_b\times u_f + u_c\times u_f
=(bc+bf+cf,\,0,\,0).
\]
Analogously, for $F_2$ the successive edge vectors are
$(a,0,0)$, $(0,0,c)$, $(-e,0,e)$ (and negatives), hence
\[
\mathbf V_{2}=(0,\,-(ac+ae+ce),\,0).
\]
For $F_3$ the successive edge vectors are
$(a,0,0)$, $(0,b,0)$, $(-d,d,0)$ (and negatives), hence
\[
\mathbf V_{3}=(0,\,0,\,-(ab+ad+bd)).
\]
Finally, for $F_0$ the successive edge vectors are
$(-d,d,0)$, $(-e,0,e)$, $(0,-f,f)$ (and negatives), so
\[
\mathbf V_{0}=(de+df+ef)(1,1,1).
\]

\smallskip
\noindent\emph{Final synthesis.}
Let now $u_{12}=(1,1,0)$, $u_{13}=(1,0,1)$, $u_{23}=(0,1,1)$ and $u_0=(1,1,1)$.
Up to sign we have
\[
\mathbf V_{12}=cd\,u_{12},\quad \mathbf V_{13}=be\,u_{13},\quad \mathbf V_{23}=af\,u_{23},
\]
\[
\mathbf V_1=(bc+bf+cf)e_1,\quad \mathbf V_2=(ac+ae+ce)e_2,\quad
\mathbf V_3=(ab+ad+bd)e_3,\quad \mathbf V_0=(de+df+ef)u_0.
\]
Therefore $Q_*=(\text{scalar})^2\,(u^TAu)$ with the corresponding $u$.
Using \eqref{eq:A} one computes
\[
u_{12}^TAu_{12}=a+b+e+f,\qquad
u_{13}^TAu_{13}=a+c+d+f,\qquad
u_{23}^TAu_{23}=b+c+d+e,
\]
\[
e_1^TAe_1=a+d+e,\quad e_2^TAe_2=b+d+f,\quad e_3^TAe_3=c+e+f,\quad
u_0^TAu_0=a+b+c.
\]
Since $a,\dots,f\ge0$, all scalar prefactors are nonnegative and hence
$\sqrt{Q_*}$ equals the corresponding product.
Summing the seven terms and substituting into
$F=\frac{2}{(\det A)^{5/6}}\sum\sqrt{Q_*}$ yields  
\begin{equation}
\label{eq:Fclosed}
\begin{aligned}
F(a,b,c,d,e,f)=\frac{2}{(\det A)^{5/6}}\Big(&
a f \sqrt{b+c+d+e}
+ b e \sqrt{a+c+d+f}
+ c d \sqrt{a+b+e+f} \\
&+ \sqrt{a+b+c}\,(de+df+ef)
+ \sqrt{a+d+e}\,(bc+bf+cf) \\
&+ \sqrt{b+d+f}\,(ac+ae+ce)
+ \sqrt{c+e+f}\,(ab+ad+bd)
\Big).
\end{aligned}
\end{equation}

\section{Analysis of parallelohedra}


\subsection{The truncated octahedron}\label{subsec:bcc}
The BCC lattice corresponds to 
\[
\boldsymbol{\rho}_{\mathrm{BCC}}=(1,1,1,1,1,1), \qquad
F(1,1,1,1,1,1) = \frac{3 (1+2\sqrt{3})}{4^{\frac 23}}.
\]

\begin{theorem}[Local minimality of BCC at fixed volume]\label{thm:bcc-local-min}
Let $F$ be defined by \eqref{eq:Fdef}. Then $\boldsymbol{\rho}_{\mathrm{BCC}}$ is a strict local minimiser of $F$
on the hypersurface
\[
\mathcal{M}=\{\boldsymbol{\rho}\in\mathbb{R}_{\ge 0}^6:\det A(\boldsymbol{\rho})=\det A(\boldsymbol{\rho}_{\mathrm{BCC}})=16\}.
\]
Equivalently, among lattices sufficiently close to BCC and with the same volume, BCC minimises the scale-invariant quotient $F$.
\end{theorem}

\begin{proof}
The natural $S_4$-action permuting the indices $\{0,1,2,3\}$ permutes the six parameters $\rho_{ij}$ and leaves $F$ invariant.
Hence at $\boldsymbol{\rho}_{\mathrm{BCC}}$ all partial derivatives $\partial_{\rho_{ij}}F$ coincide.
Moreover, $F$ is homogeneous of degree $0$, so that
\[
\sum_{i<j}\rho_{ij}\,\partial_{\rho_{ij}}F(\boldsymbol{\rho})=0.
\]
Evaluating at $\boldsymbol{\rho}_{\mathrm{BCC}}$ yields $6\,\partial_{\rho_{ij}}F(\boldsymbol{\rho}_{\mathrm{BCC}})=0$, so $\nabla F(\boldsymbol{\rho}_{\mathrm{BCC}})=0$.

At $\boldsymbol{\rho}_{\mathrm{BCC}}$ the Hessian matrix $\mathbf{H}$ has the $S_4$--invariant form
\[
\mathbf{H}=\frac{2^{2/3}}{768}\,
\begin{pmatrix}
\alpha & \beta & \beta & \beta & \beta & \delta \\
\beta & \alpha & \beta & \beta & \delta & \beta \\
\beta & \beta & \alpha & \delta & \beta & \beta \\
\beta & \beta & \delta & \alpha & \beta & \beta \\
\beta & \delta & \beta & \beta & \alpha & \beta \\
\delta & \beta & \beta & \beta & \beta & \alpha
\end{pmatrix},
\qquad
\begin{aligned}
\alpha&=14+24\sqrt3,\\
\beta&=-25+12\sqrt3,\\
\delta&=86-72\sqrt3.
\end{aligned}
\]
Its spectrum is
\begin{eqnarray*}
&\operatorname{spec}(\mathbf H_{\mathrm{BCC}})
 = \left\{ 
0,\;
 \frac{2^{2/3}(25-12\sqrt3)}{128},\;
 \frac{2^{2/3}(25-12\sqrt3)}{128}, \frac{2^{2/3}(-3+4\sqrt3)}{32},\;
\frac{2^{2/3}(-3+4\sqrt3)}{32},\;
\frac{2^{2/3}(-3+4\sqrt3)}{32}
\right\}.
\end{eqnarray*}
The zero eigenvalue corresponds to the scaling direction $(1,1,1,1,1,1)$,
hence $\mathbf{H}\,\boldsymbol{\rho}_{\mathrm{BCC}}=0$.

Consider now the function $g(\boldsymbol{\rho})=\det A(\boldsymbol{\rho})$.
From \eqref{eq:A} one computes
\[
\nabla g(\boldsymbol{\rho}_{\mathrm{BCC}}) = (8,8,8,8,8,8)=8\,\boldsymbol{\rho}_{\mathrm{BCC}}.
\]
Therefore the tangent space of the level set $\mathcal{M}$ at $\boldsymbol{\rho}_{\mathrm{BCC}}$ is
\[
T_{\boldsymbol{\rho}_{\mathrm{BCC}}}\mathcal{M}=\{v\in\mathbb{R}^6: v\cdot \boldsymbol{\rho}_{\mathrm{BCC}}=0\}
=\boldsymbol{\rho}_{\mathrm{BCC}}^{\perp}.
\]
Letting $0\ne v\in T_{\boldsymbol{\rho}_{\mathrm{BCC}}}\mathcal{M}$,
since $\mathbf{H}$ is diagonalizable with eigenvalues $\lambda_2,\dots,\lambda_6>0$ on $\boldsymbol{\rho}_{\mathrm{BCC}}^{\perp}$,
we have $v^T\mathbf{H}v>0$, hence
$\boldsymbol{\rho}_{\mathrm{BCC}}$ is a strict local minimiser of $F$ on $\mathcal{M}$.
\end{proof}

\begin{remark}
The BCC point lies in the interior of the parameter space ($\rho_{ij}>0$). In contrast, the Simple Cubic (SC) and the Face‑Centered Cubic (FCC) lattices lie on the boundary, where some $\rho_{ij}=0$. At these points the Voronoi cell degenerates: SC gives a cube (eight hexagonal faces collapse to points), and FCC gives a rhombic dodecahedron (six quadrilateral faces collapse to edges). 
\end{remark}

\subsection{The rhombic dodecahedron}\label{sec:dodeca}

The FCC lattice corresponds to
\[
\boldsymbol{\rho}_{\mathrm{FCC}}=(0,1,1,1,1,0), \qquad
F(\boldsymbol{\rho}_{\mathrm{FCC}})=3\cdot 2^{5/6}.
\]

\begin{theorem}[Local behaviour of FCC]\label{thm:fcc-local}
Let $F$ be defined by \eqref{eq:Fdef}. Then $\boldsymbol{\rho}_{\mathrm{FCC}}$ is a stationary point of $F$ in $\mathbb R_{\ge 0}^6$, but it is not a local minimiser. More precisely, $\boldsymbol{\rho}_{\mathrm{FCC}}$ is a saddle point of $F$.

On the boundary stratum
\[
\Sigma_{\mathrm{RD}}=\{\boldsymbol{\rho}\in\mathbb R_{\ge 0}^6:\ a=f=0\},
\]
which corresponds to rhombic dodecahedra, $\boldsymbol{\rho}_{\mathrm{FCC}}$ is a strict local minimiser of $F$ at fixed volume.
\end{theorem}

\begin{proof}
A direct computation from \eqref{eq:Fclosed} gives
\[
\nabla F(\boldsymbol{\rho}_{\mathrm{FCC}})=0.
\]
The Hessian of $F$ at $\boldsymbol{\rho}_{\mathrm{FCC}}$ is
\[
\mathbf H_{\mathrm{FCC}}
:=\nabla^2F(\boldsymbol{\rho}_{\mathrm{FCC}})
=\frac{2^{5/6}}{192}
\begin{pmatrix}
56 & 0 & 0 & 0 & 0 & 96\sqrt2-220 \\
0 & 27 & -9 & -9 & -9 & 0 \\
0 & -9 & 27 & -9 & -9 & 0 \\
0 & -9 & -9 & 27 & -9 & 0 \\
0 & -9 & -9 & -9 & 27 & 0 \\
96\sqrt2-220 & 0 & 0 & 0 & 0 & 56
\end{pmatrix},
\]
with spectrum 
\[
\operatorname{spec}(\mathbf H_{\mathrm{FCC}})
=
\Bigl\{
-\tfrac{2^{5/6}}{48}(41-24\sqrt2),\;
0,\;
\tfrac{3}{16}2^{5/6},\;
\tfrac{3}{16}2^{5/6},\;
\tfrac{3}{16}2^{5/6},\;
\tfrac{2^{5/6}}{48}(69-24\sqrt2)
\Bigr\}.
\]
The zero eigenvalue corresponds to the scaling direction $(0,1,1,1,1,0)$,
and the negative eigenvalue $\lambda_-$ corresponds to the direction
\[
v_-=(1,0,0,0,0,1).
\]
Therefore, for $t>0$ sufficiently small, we have
\[
F(\boldsymbol{\rho}_{\mathrm{FCC}}+t\,v_-)
=
F(\boldsymbol{\rho}_{\mathrm{FCC}})
+\frac12\,\lambda_-\,t^2+o(t^2)
<
F(\boldsymbol{\rho}_{\mathrm{FCC}}),
\]
so that $\boldsymbol{\rho}_{\mathrm{FCC}}$ is not a local minimiser of $F$; it is a saddle point.

We now restrict to the stratum $a=f=0$ and define
\[
F_{\mathrm{RD}}(b,c,d,e):=F(0,b,c,d,e,0).
\]
Then
\[
F_{\mathrm{RD}}(1,1,1,1)=3\cdot 2^{5/6},
\qquad
\nabla F_{\mathrm{RD}}(1,1,1,1)=0,
\]
and the restricted Hessian is
\[
\nabla^2F_{\mathrm{RD}}(1,1,1,1)
=
\frac{3\cdot 2^{5/6}}{64}
\begin{pmatrix}
3&-1&-1&-1\\
-1&3&-1&-1\\
-1&-1&3&-1\\
-1&-1&-1&3
\end{pmatrix}.
\]
Its spectrum is
\[
\operatorname{spec}\bigl(\nabla^2F_{\mathrm{RD}}(1,1,1,1)\bigr)
=
\Bigl\{
0,\;
\tfrac{3}{16}2^{5/6},\;
\tfrac{3}{16}2^{5/6},\;
\tfrac{3}{16}2^{5/6}
\Bigr\}.
\]
As above, the zero eigenvalue is the scaling direction $(1,1,1,1)$,
and the restricted Hessian is positive definite on the fixed-volume tangent space; hence FCC is a strict local minimiser within stratum $a=f=0$.
\end{proof}

\subsection{The cube}

The simple cubic (SC) lattice corresponds to
\[
\boldsymbol{\rho}_{\mathrm{SC}}=(1,1,1,0,0,0),
\qquad
F(\boldsymbol{\rho}_{\mathrm{SC}})=6.
\]
This is a highly degenerate configuration where eight hexagonal faces of the generic Voronoi cell have collapsed to the vertices of the cube. Increasing the parameters $d, e,$ and $f$ introduces non-orthogonal lattice vectors, which physically truncates the corners of the cubic cell. This causes the degenerate vertices to open up into actual hexagonal faces, smoothing the cell and strictly decreasing the isoperimetric quotient.

\begin{theorem}[Local behaviour of SC]
Let $F$ be defined by  \eqref{eq:Fdef}. Then
$\boldsymbol{\rho}_{\mathrm{SC}}$ is not a stationary point of $F$ in
$\mathbb R_{\ge 0}^6$.

On the boundary stratum
\[
\Sigma_{\mathrm{box}}=\{\boldsymbol{\rho}\in\mathbb R_{\ge 0}^6:\ d=e=f=0\},
\]
which corresponds to orthogonal lattices (rectangular boxes),
$\boldsymbol{\rho}_{\mathrm{SC}}$ is a strict local minimiser of $F$ at fixed volume.
\end{theorem}

\begin{proof}
Using \eqref{eq:Fclosed} one finds the exact gradient
\begin{equation}
\label{eq:gradSC}
\nabla F(\boldsymbol{\rho}_{\mathrm{SC}})
=\bigl(0,0,0,\,-4+2\sqrt2,\,-4+2\sqrt2,\,-4+2\sqrt2\bigr).
\end{equation}
Since $-4+2\sqrt2<0$, the cube is \emph{not} a stationary point of $F$: increasing any of $d,e,f$ decreases $F$.

If we restrict to the boundary stratum $d=e=f=0$ (orthogonal lattices, whose Voronoi cells are rectangular boxes), the restricted functional
\begin{equation}
\label{eq:Fbox}
F_{\mathrm{box}}(a,b,c) := F(a,b,c,0,0,0)
= \frac{2\bigl(\sqrt{a}\,bc+a\sqrt{b}\,c+ab\sqrt{c}\bigr)}{(abc)^{5/6}}
\end{equation}
satisfies $\nabla F_{\mathrm{box}}(1,1,1)=0$ and
\[
\nabla^2 F_{\mathrm{box}}(1,1,1)=\frac{1}{6}
\begin{pmatrix}
2 & -1& -1\\
-1 &2 & -1\\
-1 & -1 &2
\end{pmatrix},
\qquad
\operatorname{spec}\bigl(\nabla^2 F_{\mathrm{box}}(1,1,1)\bigr)=\Bigl\{0,\frac12,\frac12\Bigr\}.
\]
The kernel corresponds to the scaling direction $(1,1,1)$; hence on the fixed-volume constraint $abc=\mathrm{const}$ the cube is a strict local minimiser of $F_{\mathrm{box}}$.
\end{proof}

\subsection{Comparison with numerical simulations}\label{sec:lucarini}

In \cite{Lucarini2008} the author studies three-dimensional Voronoi tessellations generated by crystals perturbed by Gaussian noise, and reports the \emph{isoperimetric quotient}
\[
Q=\frac{36\pi V^2}{A^3},
\]
where $V$ and $A$ are the volume and surface area of a cell (so $Q=1$ for a sphere). In our notation $F=A/V^{2/3}$, hence
\[
Q=\frac{36\pi}{F^3}.
\]
Table~\ref{tab:lucarini_comparison} compares our exact values with the values reported in \cite{Lucarini2008}.
\begin{table}[h!]
\centering
\begin{tabular}{|l|c|c|c|c|}
\hline
\textbf{Structure} & \textbf{Value of $F$} & \textbf{Value of $Q$} & \textbf{$Q$ computed in \cite{Lucarini2008}} \\
\hline
SC (cube) & $6$ & $0.5236$ & $0.5236$ \\
\hline
FCC (rhombic dodecahedron) & $3\cdot 2^{5/6}$ & $0.7405$ & $0.7405$ \\
\hline
BCC (truncated octahedron) & $\dfrac{3\cdot 2^{2/3}(1+2\sqrt3)}{4}$ & $0.7534$ & $0.7534$  \\
\hline
\end{tabular}
\smallskip
\caption{Isoperimetric quotients of SC, FCC and BCC.}
\label{tab:lucarini_comparison}
\end{table}


In \cite{Lucarini2008} it is shown that, for small noise intensity, the ensemble-average of $Q$ has a quadratic \emph{maximum} at the FCC and BCC crystals. Since $Q=36\pi/F^3$, this corresponds to $F$ having a quadratic \emph{minimum} at the same points. From our analysis it follows that
\begin{itemize}
\item Theorem~\ref{thm:bcc-local-min} proves that BCC is a strict local minimum of $F$;
\item in Section~\ref{sec:dodeca} we show that FCC is a saddle point for $F$ in the six-dimensional parameter space, while remaining a local minimizer within a suitable four-parameter family of lattices;
\item the gradient \eqref{eq:gradSC} already shows instability of $SC$, in agreement with the immediate topological changes numerically observed under noise.
\end{itemize}

\section{Global minimality of BCC on a family of lattices}
\label{sec:symmetry}

\subsection{Two-parameters families of lattices}

We now consider families of lattices whose six Selling parameters take only two values, say $p$ and $q$. Observe that we may encode the Selling parameters as the six edges of the complete graph $K_4$ on $\{0,1,2,3\}$: the edge $ij$ carries the label $\rho_{ij}$. Therefore a two-parameters family is  determined by the subset of edges on which the value $p$ is assigned, while the complementary edges carry the value $q$.

The symmetric group $S_4$ acts by relabelling the vertices of the tetrahedron, hence on the edges of $K_4$ that is if $\sigma\in S_4$,  \[
  (\sigma\cdot\boldsymbol{\rho})_{ij}:=\rho_{\sigma(i)\sigma(j)}.
\] Since $F$ is invariant under this relabelling, it is enough to study one representative for each $S_4$-orbit.

\begin{proposition}\label{prop:symmetry-orbits}
Let $\sigma\in S_4$, then the following hold:
\begin{enumerate}
\item[\textup{(i)}] For every $\boldsymbol{\rho}\in\R_{\ge 0}^6$, with $\det A(\boldsymbol{\rho})>0$, we have
\[
  F(\sigma\cdot\boldsymbol{\rho})=F(\boldsymbol{\rho}).
\]
\item[\textup{(ii)}] Among the $6$ strata of type $(1,5)$ there is a single $S_4$--orbit, represented by
\[
  \mathcal C:=\{(p,q,q,q,q,q)\}.
\]
\item[\textup{(iii)}] Among the $15$ strata of type $(2,4)$ there are exactly two $S_4$--orbits:
\begin{itemize}
  \item the \emph{opposite} orbit, represented by
  \[
    \mathcal O:=\{(p,q,q,q,q,p)\};
  \]
  \item the \emph{adjacent} orbit, represented by
  \[
    \mathcal A:=\{(p,p,q,q,q,q)\}.
  \]
\end{itemize}
\item[\textup{(iv)}] Among the $20$ strata of type $(3,3)$ there are exactly three $S_4$--orbits:
\begin{itemize}
  \item the \emph{star} orbit, represented by
  \[
    \mathcal S:=\{(p,p,p,q,q,q)\};
  \]
  \item the \emph{triangle} orbit, represented by
  \[
    \mathcal T:=\{(p,p, q, p,q,q)\};
  \]
  \item the \emph{path} orbit, represented by
  \[
    \mathcal P:=\{(p,q,q,p,q,p)\}.
  \]
\end{itemize}
\end{enumerate}
\end{proposition}

\begin{proof}
Point \textup{(i)} follows from the tetrahedral symmetry of the Selling description: relabelling the obtuse superbase conjugates the Gram matrix by a permutation matrix, hence preserves $\det A$, and it permutes the seven terms appearing in the closed formula \eqref{eq:Fclosed}. Therefore $F$ is unchanged.

For \textup{(ii)}, all $(1,5)$ families are equivalent because $S_4$ acts transitively on the six edges of $K_4$.

For \textup{(iii)}, a pair of edges is either disjoint or adjacent. These two possibilities are preserved by graph automorphisms, and $S_4$ is transitive on each type. This gives the representatives $(p,q,q,q,q,p)$ and $(p,p,q,q,q,q)$.

For \textup{(iv)}, a $3$-edge subgraph of $K_4$ has degree sequence $(3,1,1,1)$, $(2,2,2,0)$, or $(1,2,2,1)$, corresponding respectively to a star, a triangle, or a path of length $3$. These degree sequences are invariant under relabelling, and $S_4$ is transitive within each class. This yields the representatives $(p,p,p,q,q,q)$, $(p,p, q, p,q,q)$, and $(p,q,q,p,q,p)$.
\end{proof}

Thus, up to symmetry, the analysis reduces to the five representatives
\[
  \mathcal C,\qquad \mathcal O,\qquad \mathcal A,\qquad \mathcal S,\qquad \mathcal P,
\]
because the triangle orbit is obtained from the star orbit by exchanging $p$ and $q$.

\subsection{\texorpdfstring{The opposite orbit}{The opposite orbit}}\label{sec:24}

In the following, we shall focus our attention to the orbit
\[
  \mathcal O=\{(p,q,q,q,q,p)\},
\]
which contains both BCC and FCC lattices.

We prove the following result.
\begin{theorem}\label{thm:main-T}
For every $p,q\geq 0$,
\[
  F(p,q,q,q,q,p)\;\ge\;F(1,1,1,1,1,1),
\]
with equality if and only if $p=q$.
\end{theorem}
\begin{proof} 
Substituting $a=f=p$ and $b=c=d=e=q$ into \eqref{eq:Fclosed} gives
\[
\det A=4q(p+q)^2
\]
and  
\[
F(p,q,q,q,q,p)=\frac{2}{(\det A)^{5/6} }\left(2p^2\sqrt q + 2q^2\sqrt{2(p+q)} + 4q(q+2p)\sqrt{p+2q}\right).
\]
Note that as $q\to 0^+$, with $p>0$, it holds $F(p,q,q,q,q,p)\to +\infty$. 
We assume then $q>0$ and set $u=p/q\ge 0$: therefore $F$ reduces to  
\[
  F(p,q,q,q,q,p)
  \,=\,
  \frac{2^{1/3}\,H(u)}{(1+u)^{5/3}}\,=:\,\widetilde F(u),
\]
where
\[
  H(u)
  :=
  u^2 + \sqrt{2}\,\sqrt{1+u} + 2(1+2u)\sqrt{u+2}.
  \tag{H}
\]

Notice that
\begin{itemize}
  \item $\displaystyle \widetilde F(1)
    =\frac{3\cdot 2^{2/3}(1+2\sqrt{3})}{4}=F_{\BCC}$;
  \item $\widetilde F(0)
    =3\cdot 2^{5/6}=F_{\mathrm{FCC}}$;
  \item $\displaystyle \lim_{u\to+\infty}\widetilde F(u)=+\infty$.
\end{itemize}
 
Therefore, to conclude the result it is sufficient to prove that 
$\widetilde F$ is strictly decreasing on $(0,1)$ and strictly increasing on $(1,\infty)$, so it has a unique   global minimum at $u=1$, i.e. at $p=q$.

  Define  \[
  \psi(u):=3(1+u)H'(u)-5H(u),\qquad u\ge0.
\]
Then
\[
  \widetilde F'(u)=\frac{2^{1/3}}{3}\,\psi(u)\,(1+u)^{-8/3}.
\]
Since  $\operatorname{sgn}\widetilde F'(u)=\operatorname{sgn}\psi(u)$ for all $u\ge 0$, we are reduced to study the sign of $\psi(u)$. By direct computation one easily checks that $\psi(0)=\psi(1)=0$. 

We will show that  $\psi(u)<0$ for $0<u<1$ and $\psi(u)>0$ for $u>1$. It then follows that $\widetilde F$ is strictly decreasing on $(0,1)$ and strictly increasing on $(1,\infty)$, so that $F(p,q,q,q,q,p)$ has a global minimum at $p=q$, i.e. at  $(1,1,1,1,1,1)$.
 
Observe that  $\psi'(u)= 3(1+u)H''(u)-2H'(u)$ and $\psi''(u)=H''(u)+3(1+u)H'''(u)$, so
  by direct computation
\begin{align*} H'(u)&= 2u+ \frac{\sqrt{2}}{2\sqrt{1+u}}+ \frac{9+6u}{\sqrt{u+2}}\\
H''(u)&=2-\frac{ \sqrt{2}}{4(1+u)^{3/2}} +\frac{3(2u+5)}{2(u+2)^{3/2}}
\\ H'''(u)&= \frac{3\sqrt{2}}{8(1+u)^{5/2}} -\frac{3(2u+7)}{4(u+2)^{5/2}}.\end{align*}
we get \[ \psi''(u)
=
2+\frac{7\sqrt2}{8(u+1)^{3/2}}-
\frac34\,\frac{2u^2+9u+1}{(u+2)^{5/2}}.
\]
Letting $t:=u+2\geq 2$, we get
\[\frac34\,\frac{2u^2+9u+1}{(u+2)^{5/2}}=\frac{3}{4t^{5/2}} (2t^2+t-9)\leq \frac{3}{2t^{1/2}}+\frac{3}{4 t^{3/2}}\leq \frac{3}{2\sqrt{2}}+\frac{3}{8\sqrt{2}}=\frac{15}{8\sqrt{2}}
\]
whence
\[
\psi''(u)\ge 2-\frac{15}{8\sqrt2}>0.
\]
Therefore, the function $\psi$ is strictly convex and smooth for $u\geq 0$, with  $\psi(0)=\psi(1)=0$. Moreover $\psi'(1)>0$. 
It follows that $\psi(u)<0$ for $0<u<1$ and $\psi(u)>0$ for $u>1$, which gives the thesis.
\end{proof}



\begin{thebibliography}{99}

\bibitem{bezdek_sphere_packings}
Bezdek, K. Sphere packings revisited. \textit{European Journal of Combinatorics} 27(6), 864-883, 2006.

\bibitem{B}
Barnes, E. S. The covering of space by spheres. \textit{Canad. J. Math.} 8, 293-304, 1956.

\bibitem{BS}
Barnes, E. S. \& Sloane, N. J. A. The optimal lattice quantizer in three dimensions. 
\textit{SIAM J. Algebraic Discrete Methods} 4(1), 30-41, 1983.

\bibitem{CN}
Cesaroni, A. \& Novaga, M. Minimal periodic foams with equal cells.
A. Cesaroni and M. Novaga,
in ``Anisotropic Isoperimetric Problems and Related Topics",
\textit{Springer INdAM Series} 62, 15-24, 2024.

\bibitem{conway_sloane_voronoi}
Conway, J. H. \& Sloane, N. J. A. On the Voronoi Regions of Certain Lattices. 
\textit{SIAM Journal on Algebraic and Discrete Methods}, 5(3), 294-305, 1984.

\bibitem{Langi}
Zsolt, L. An isoperimetric problem for three-dimensional parallelohedra.
\textit{Pacific J. Math.} 316(1), 169-181, 2022.

\bibitem{Lucarini2008}
Lucarini, V. Three-dimensional Random Voronoi Tessellations: From Cubic Crystal Lattices to Poisson Point Processes. 
\textit{Journal of Statistical Physics} 134, 185-206, 2009.

\end{thebibliography}
\end{document}